\documentclass[12pt]{amsart}
\makeatletter 
\@addtoreset{thm}{section} 
\makeatother

\usepackage{mathrsfs, amssymb, array}
\usepackage{mathcomp}

\newcommand{\type}[1]{$\mathrm{(#1)}$}
\newcommand{\ov}[1]{\overline{#1}}

\newcommand{\comment}[1]{}

\newcommand{\LLL}{{\mathscr{L}}}
\newcommand{\MMM}{{\mathscr{M}}}

\newcommand{\OOO}{\mathscr{O}}

\newcommand{\CC}{\mathbb{C}}
\newcommand{\ZZ}{\mathbb{Z}}
\newcommand{\PP}{\mathbb{P}}
\newcommand{\QQ}{\mathbb{Q}}

\newcommand{\toplus}{\mathbin{\tilde\oplus}}
\newcommand{\totimes}{\mathbin{\tilde\otimes}}
\newcommand{\xref}[1]{{\rm \ref{#1}}}

\newcommand{\comp}{\mathbin{\scriptstyle{\circ}}}
\newcommand{\ldeg}{\operatorname{ldeg}}
\newcommand{\Spec}{\operatorname{Spec}}
\newcommand{\wt}{\operatorname{wt}}
\newcommand{\siz}{\operatorname{siz}}
\newcommand{\qldeg}{\operatorname{qldeg}}
\newcommand{\Sing}{\operatorname{Sing}}
\newcommand{\TL}{\operatorname{TL}}
\newcommand{\red}{\operatorname{red}}
\newcommand{\gr}{\operatorname{gr}}
\newcommand{\ql}{\operatorname{ql}}
\newcommand{\ord}{\operatorname{ord}}
\newcommand{\mt}[1]{\operatorname{#1}}

\newcommand{\muu}{\mbox{\boldmath $\mu$}}

\newtheorem{thm}{Theorem}

\newtheorem{corollary}[thm]{Corollary}
\newtheorem{lemma}[thm]{Lemma}
\newtheorem{prop}[thm]{Proposition}
\newtheorem{cmp}[thm]{Computation}

\theoremstyle{definition}

\newtheorem{say}[thm]{}

\newtheorem{remark}[thm]{Remark}
\newtheorem{notation}[thm]{Notation}
\newtheorem{case}[equation]{Case}
\newtheorem{substatement}[equation]{}

\newtheorem{step}[equation]{Step}
\newtheorem{ppar}[equation]{}

\makeatletter
\@addtoreset{equation}{section}
\@addtoreset{equation}{thm}
\makeatother

\title{On $\QQ$-conic bundles, III}
\author{Shigefumi Mori}
\author{Yuri Prokhorov}
\dedicatory{To the memory of the late Professor Masayoshi Nagata}
\thanks {
The research of the first author was supported by JSPS Grant-in-Aid for
Scientific Research (B)(2), Nos. 16340004 and 20340005.
The second author was
partially supported by grants 
RFBR, \textnumero\ 08-01-00395-a and 06-01-72017-MHTI-a.}
\address{Shigefumi Mori: RIMS, 
Kyoto University, Oiwake-cho, Kitashirakawa, Sakyo-ku, Kyoto
606-8502, Japan}
\email{mori@kurims.kyoto-u.ac.jp}
\address{Yuri Prokhorov: Department 
of Algebra, Faculty of Mathematics, Moscow State
University, Moscow 117234, Russia}
\email{prokhoro@mech.math.msu.su}

\begin{document}
\maketitle

\begin{abstract}
A $\mathbb Q$-conic bundle germ is a proper morphism from a threefold
with only terminal singularities to the germ $(Z \ni o)$ of a normal 
surface such that fibers are connected and the anti-canonical 
divisor is relatively ample. Building upon our previous paper
\cite{Mori-Prokhorov-2008},
we prove the existence of a Du Val anti-canonical member
under the assumption that the central fiber is irreducible.
\end{abstract}

\section{Introduction}
The present paper is a continuation of a series of 
papers \cite{Mori-Prokhorov-2008}, 
\cite{Mori-Prokhorov-2006a}.

Recall that 
a \textit{$\QQ$-conic bundle} is a projective morphism
$f\colon X \to Z$ from an (algebraic or analytic) 
threefold with terminal singularities to a
surface that satisfies the following properties:
\begin{enumerate}
\item
$f_*\OOO_X=\OOO_Z$
and 
all fibers are one-dimensional,
\item
$-K_X$ is $f$-ample.
\end{enumerate}
For $f\colon X\to Z$ as above and for a point $o\in Z$, we call the
\textit{analytic} 
germ $(X, f^{-1}(o)_{\red})$ a \textit{$\QQ$-conic
bundle germ}.

Out main result is the following
\begin{thm}[{cf. \cite[(2.2)]{Kollar-Mori-1992}}]
\label{(2.2)}
Let $f\colon (X, C\simeq \PP^1) \to (Z, o)$ be a 
$\QQ$-conic bundle germ with smooth base surface $Z$. 
Assume that $f$ is of type \type{IC}, \type{IIB}, \type{IA}$+$\type{IA}, or
\type{IA}$+$\type{IA}$+$\type{III}. Then 
a general member $E_{X}$ of $|-K_{X}|$ and 
$E_{Z} :=\Spec_Z f_*\OOO_{E_X}$ have only 
Du Val singularities. To be more explicit, 
the minimal resolutions of $E_{Z}$ and $E_{X}$ coincide.
We have the the following possibilities depending on the 
type of $(X, C)$ 
\textup(below, $o' \in E_Z$ is the image of $C$\textup):

\begin{case}[{\type{IC}, \cite[(2.2.2)]{Kollar-Mori-1992}}]
\label{(2.2.2)}
$(E_{Z},o')$ is $D_{m}$ and $\Delta (E_{Z},o')$ is 
\[
\begin{array}{ll}
&\circ \\ 
&\mid \\ 
\underbrace{\circ -\cdots - \circ}_{m-3} -&\circ - \bullet, 
\end{array}
\]
where $m$, the index of the \type{IC} 
point of $C$, is odd and $m\ge 5$.
\end{case}

\begin{case}[\type{IIB}, {\cite[(2.2.2${}'$)]{Kollar-Mori-1992}}]
\label{2.2.2')}
$(E_{Z},o')$ is $E_{6}$ and $\Delta (E_{Z},o')$ is 
\[
\begin{array}{ll} 
&\circ \\ 
&\mid \\ 
\circ - \circ -&\circ - \circ - \bullet.
\end{array}
\]
\end{case}

\begin{case}[{\type{IA}$+$\type{IA}, 
\cite[(2.2.3)]{Kollar-Mori-1992},
\cite{Mori-2007}}]
\label{(2.2.3)}
The two \type{IA} points are an 
ordinary point of odd index $m\ge 3$ and 
an index $2$ point of type $cA/2$, $cAx/2$ or $cD/2$ with
axial multiplicity $k$ such that $k \ge 2$ if $m=3$. 
$(E_{Z},o')$ is $D_{2k+m}$, 
$\Sing E_{X}$ is 
$A_{m-1}+D_{2k}$ ($A_{m-1}+A_{1}+A_{1}$ if $k=1$) and 
$\Delta (E_{Z},o')$ 
is 
\[
\begin{array}{ll}
&\circ \\ 
&\mid \\ 
\underbrace{\circ -\cdots - \circ}_{m-1} - \bullet - \circ - \cdots -&\circ - \circ
\end{array}.
\]
\end{case}

\begin{case}[{\type{IA}$+$\type{IA}$+$\type{III}, \cite[(2.2.3${}'$)]{Kollar-Mori-1992}}]
\label{(2.2.3')}
The two \type{IA} points are both ordinary and of indices 2 and $m$ (odd, $\ge 3$). 
$(E_{Z},o')$ is $D_{m+2}$ ($m\ge 3$) or
$E_6$ ($m=3$). The graph $\Delta (E_{Z},o')$ is 
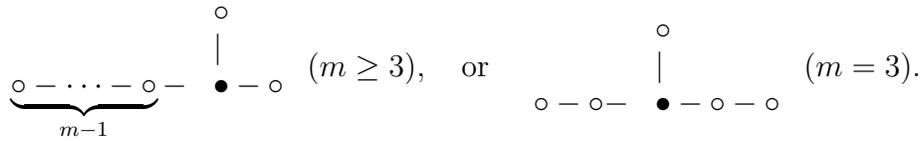
\[
\begin{array}{ll}
&\circ \\ 
&\mid \\ 
\underbrace{\circ -\cdots - \circ}_{m-1} -&\bullet - \circ
\end{array}\ (m\ge 3), \quad \text{or}\quad
\begin{array}{ll}
&\circ \\ 
&\mid \\ 
\circ - \circ -&\bullet - \circ - \circ
\end{array}\ (m=3).
\]
\end{case}
Above, we use the usual notation of graphs $\Delta (E_{Z},o')$:
$\bullet$ corresponds to the curve $C$ and each
$\circ$ corresponds to a $(-2)$-curve on the minimal
resolution of $E_X$.
\end{thm}

\begin{remark}
The following are the classifications of $\QQ$-conic bundles over
a smooth base \cite{Mori-Prokhorov-2008}:
\begin{enumerate}
\item $\emptyset$, \type{III}, \type{III}$+$\type{III},
\item \type{IA}, \type{IA}$+$\type{III}, \type{IA}$+$\type{III}$+$\type{III},
\item \type{IIA}, \type{IIA}$+$\type{III},
\item \type{IC}, \type{IIB},
\item \type{IA}$+$\type{IA} of indices $2$ and odd $m\ (\ge 3)$,
\item \type{IA}$+$\type{IA}$+$\type{III} of indices $2$, odd $m\ (\ge 3)$ and $1$.
\end{enumerate}
\end{remark}

We recall that in the cases (ii) and (iii)
(the case (i) is trivial) 
a general member $E_X\in |-K_X|$ does not contain 
$C$ and has only Du Val singularity at $E_X\cap C$ 
\cite[(1.3.7)]{Mori-Prokhorov-2008}.
Cases (iv)-(vi) are treated in this paper.
In the case of singular base, the existence
of a Du Val member $E_X\in |-K_X|$ follows from 
\cite[(1.3.7)]{Mori-Prokhorov-2008}.
Thus we have the following

\begin{corollary}[Reid's general elephant conjecture]
\label{general-elephant}
Let $f\colon (X, C\simeq\PP^1) \to (Z, o)$ 
be a $\QQ$-conic bundle germ. Then 
a general member $E_{X}$ of $|-K_{X}|$ has only 
Du Val singularities.
\end{corollary}

The techniques used in this paper is very similar to 
that in \cite[\S 2]{Kollar-Mori-1992}. 
The main difference is that for the conic bundle case 
we have no vanishing of $H^1(X,\omega_X)$
which was used in \cite[(2.5)]{Kollar-Mori-1992}
to extend sections from $D\in |-2K_X|$.
Instead of vanishing we use 
Proposition \ref{(2.5)}
and Corollary \ref{surj_on_omega}.

\section{Preliminaries}
\label{preliminaries}
\begin{prop}[cf. {\cite[(2.5)]{Kollar-Mori-1992}}]
\label{(2.5)}
Let $f\colon (X, C) \to (Z,o)$ be a $\QQ$-conic bundle germ
with smooth base surface $Z$. Let $D \in |-nK_X|$ for some
integer $n>0$ such that the restriction 
$g=f_D\colon D \to Z$ is finite. The standard
exact sequence
\[
0 \to \omega_X \to \omega_X(D) \to \omega_D \to 0
\]
induces the exact sequence
\[
f_*\omega_X(D) \stackrel{\alpha}{\to} g_*\omega_D 
\stackrel{\mt{Tr}_{D/Z}}
{\longrightarrow} \omega_Z,
\]
where $\omega_Z \simeq R^1f_*\omega_X$
by \cite[Lemma (4.1)]{Mori-Prokhorov-2008}, and the natural 
map $g^*\colon \omega_Z \to g_*\omega_D$ has the
property $\mt{Tr}_{D/Z} \comp g^*=2n \mt{id}_{\omega_Z}$.
\end{prop}

\begin{proof}
In view of the relative duality, this follows from the fact that
$\OOO_Z \stackrel{}{\longrightarrow} g_*\OOO_D \stackrel{\mt{Tr}_{Z/D}}{\longrightarrow} \OOO_Z$
is the multiplication by $\deg(D/Z)=2n$. 
\end{proof}

\begin{corollary}
\label{surj_on_omega}
The homomorphism induced by Proposition \xref{(2.5)}
\[
f_*\omega_X(D) \to (g_*\omega_D)/\omega_Z
\]
is surjective.
\end{corollary}

\begin{notation}
Everywhere below 
\[
f\colon (X, C) \to (Z, o)
\]
denotes a $\QQ$-conic bundle germ. 
We assume that the curve $C$ is irreducible (and so $C\simeq \PP^1$) 
and the base surface 
$(Z,o)$ is smooth.
Notation and techniques of \cite{Mori-1988} will be used
freely.
Additionally, $(c+\sum d_iP_i^\sharp)$ denotes the element 
of 
$\mt{Cl}^{\mt{sc}}(X)\simeq \mt{Pic}^{\ell}(X)\simeq \mt{QL}(C)$ 
corresponding to $c+\sum d_iP_i^\sharp \in \mt{QL}(C)$ 
(see \cite[(2.7)]{Kollar-Mori-1992}).

The symbol $\Delta(E_Z,o')$ at the end of Theorem \ref{(2.2)}
is extended as follows.
Let $E$ be a normal surface and $C \subset E$ 
a curve such that the proper transform $\tilde{C}$ of $C$ and the exceptional
divisors $\Gamma_i$ on the minimal resolution of $E$ form a simple normal crossing
divisor. The graph $\Delta(E,C)$ is the dual graph of the divisor
$\tilde{C}+\sum \Gamma_i$, where 
each component of $\tilde{C}$ is drawn $\bullet$ and each $\Gamma_i$
is drawn $\circ$, and if no weight is specified we mean that
the corresponding $\Gamma_i$ is a $(-2)$-curve.
We note that $\Delta(E_Z,o')=\Delta(E_X,C)$.
\end{notation}

\begin{remark}
\label{rem-def}
As explained in \cite[\S 1b]{Mori-1988} and 
\cite[\S 6]{Mori-Prokhorov-2008}, any local deformation
near points $P_i\in (X,C)$ on $\QQ$-conic bundle germ 
$(X,C)$ can be extended to a global deformation 
$(X_\lambda, C_\lambda)$.
A general element $(X_\lambda, C_\lambda)$ 
of the family can be either an extremal neighborhood or
again a $\QQ$-conic bundle germ. 
In some cases this allows us to obtain certain restrictions
on the possible configurations of singular points.
We will use these arguments several times below.
\end{remark}

\section{Case of \type{IC}}
\label{(2.10)}
\begin{say}[\textbf{Cf. \cite[(2.10)]{Kollar-Mori-1992}}] 
\label{sayIC}
Let $P$ be the \type{IC} point of index $m$ and 
\[
(y_{1},y_{2},y_{4})/\muu_{m}(2,m-2,1) 
\]
be coordinates for the canonical 
cover $P^\sharp \in C^\sharp \subset X^\sharp$ given in \cite[(A.3)]{Mori-1988} so that $C^\sharp$ is 
parametrized by $(t^{2},t^{m-2},0)$. 
In this case, $P$ is the only singular 
point of $X$ on $C$ 
\cite[(8.2)]{Mori-Prokhorov-2008}. 
Since $y^{m-2}_{1} - y^{2}_{2}$ and $y_{4}$ 
generate the defining ideal of $C^\sharp$, they form an $\ell$-free $\ell$-basis of 
$\gr^{1}_{C}\OOO_{X}$. It is easy to see that 
$\Omega=dy_{1}\wedge dy_{2}\wedge dy_{4}$ is an $\ell$-free 
$\ell$-basis of 
$\gr^{0}_{C}\omega_{X}$. Then 
$\ql_{C}(\omega_{X})=-P^\sharp$ and 
$D=\{y_{1}=0\}/\muu_{m} \in |-2K_{X}|$ 
by $(D\cdot C)=2/m$. By 
\[
\ql_{C}(\gr^{0}_{C}(\omega {^*}))=\ql_{C}(\omega {^*})=-\ql_{C}(\omega)=P^\sharp, 
\]
one has 
\[
\deg(\gr^{0}_{C}(\omega {^*}))=\TL(P^\sharp)=-U(-1)=-1 
\]
(see \cite[(8.9.1)(iii)]{Mori-1988}), where 
\[
U(x)=\min\{z \in \ZZ \mid mz-x \in 2\ZZ_{+}+(m-2)\ZZ_{+}\}
\]
\cite[(2.8)]{Mori-1988}. Thus 
\[
\gr^{0}_{C}(\omega {^*})=\omega {^*}/F^{1}_{C}(\omega {^*}) \simeq \OOO_{C}(-1)
\]
and $H^{0}(\OOO_{X}(-K_{X}))=H^{0}(F^{1}_{C}(\omega {^*}))$. 
Hence a general section $s \in H^{0}(\OOO_{X}(-K_{X}))$ 
is written as $(\lambda \cdot y_{4}+\mu \cdot (y^{m-2}_{1} - y^{2}_{2}))/\Omega$ near $P$, where $\lambda \in \OOO_{X}$ 
and $\mu \in \OOO_{X^\sharp}$ with 
$\wt \mu \equiv 5 \mod (m)$.
Now apply 
Corollary \xref{surj_on_omega} with 
$n=2$. We have $\omega_D\simeq\OOO_D(-K_X)$, and 
\begin{multline}
\label{eq-IC-omega}
g^*\omega_Z \subset 
\wedge^2 \OOO_D\bigl(d y_2^m,\ dy_2y_4^2,\ 
dy_2^{(m+1)/2}y_4,\ dy_4^m\bigr)\subset
\\
\OOO_D\bigl(y_2^{(3m-1)/2},\, y_2^my_4,\, 
y_2^{(m+1)/2}y_4^2,\, y_2^{m-1}y_4^{m-1},\, 
y_2^{(m-1)/2}y_4^m,\, y_4^{m+1}\bigr) dy_2\wedge dy_4.
\end{multline}
Since $y_4dy_2\wedge dy_4$ corresponds to $y_4/\Omega$, we have $\lambda(0) \neq 0$.
Hence $s$ induces a section 
$\ov{s}$ of $\gr^{1}_{C}(\omega {^*})=F^{1}_{C}(\omega {^*})/F^{2}_{C}(\omega {^*})$ 
and $\ov{s}$ is a part of an $\ell$-free $\ell$-basis of $\gr^{1}_{C}(\omega {^*})$ at $P$. This 
induces an $\ell$-exact sequence 
\begin{equation}
\label{(2.10.1)}
0 \to (a) \to \gr^{1}_{C}(\omega {^*}) 
\to (b+5P^\sharp) \to 0, 
\end{equation}
where $a, b \in \ZZ$, $a\ge 0$. This is 
because $y_{4}/\Omega$ and 
$(y^{m-2}_{1} - y^{2}_{2})/\Omega$ have weights
$\equiv 0$ and $m-5 \mod (m)$, 
respectively. We claim an $\ell$-isomorphism 
\begin{equation}
\label{(2.10.2)}
\gr^{1}_{C}\OOO \simeq (4P^\sharp) 
\toplus (-1+(m-1)P^{\sharp}).
\end{equation}

First recall that $m$ is odd and $m\ge 5$ since $P$ is an 
\type{IC} point. By $\eqref{(2.10.1)} \totimes \gr^{0}_{C}\omega$, 
there is an $\ell$-exact sequence 
\begin{equation}
\label{(2.10.3)}
0 \to ((a-1)+(m-1)P^\sharp) 
\to \gr^{1}_{C}\OOO \to (b+4P^\sharp) \to 0. 
\end{equation}
It follows from $i_{P}(1)=2$ \cite[(6.5)]{Mori-1988} that 
$\deg \gr^{1}_{C}\OOO=-1$. By 
\[
\deg 
((a-1)+(m-1)P^\sharp)=a-1
\]
and $\deg (b+4P^\sharp)=b$ \cite[(8.9.1)(iii)]{Mori-1988}, 
we have $a+b=0$. Hence from 
\begin{multline*}
\ql_{C}((a-1)+(m-1)P^\sharp-(b+4P^\sharp)) =
\\ 
=\ql_{C}(2a-1+(m-5)P^\sharp)=2a-1\ge -1,
\end{multline*}
we see that \eqref{(2.10.3)} is $\ell$-split by \cite[(2.6)]{Kollar-Mori-1992}. Since $H^{1}(C,\gr^{1}_{C}\OOO) =0$ 
by \cite[Corollary (2.3.1)]{Mori-Prokhorov-2008}, 
we have $b\ge -1$ and hence $(a,b)=(0,0)$ or 
$(1,-1)$. Whence \eqref{(2.10.2)} follows if $(a,b)=(0,0)$ or $m=5$. 
Assuming $(a,b)=(1,-1)$ and $m\ge 7$, we will derive a contradiction. 
Now $\eqref{(2.10.1)}\totimes\omega_{X}^{\totimes2}$
gives us an $\ell$-exact sequence 
\[
0 \to (-1+(m-2)P^\sharp) 
\to \gr^{1}_{C}\omega \to (-1+3P^\sharp) \to 0.
\]
Since $\Spec \OOO_X/I_C^{(2)}\not\supset f^{-1}(o)$, by 
\cite[Theorem (4.4)]{Mori-Prokhorov-2008} we have 
$H^1(C,\gr_C^1\omega)=0$.
Whence, 
\[
-1\le \deg (-1+3P^\sharp) =\TL(-1+3P^\sharp)=-2. 
\]
This is a 
contradiction and \eqref{(2.10.2)} is proved. Thus, 
\begin{equation}
\label{(2.10.4)}
\gr^{1}_{C}(\omega {^*})=(5P^\sharp)\toplus(0). 
\end{equation}

We claim that $\ov{s}$ is a nowhere vanishing section of the locally 
free sheaf 
$\gr^{1}_{C}(\omega {^*}) \simeq \omega {^*}\totimes\gr^{1}_{C}\OOO$. In case $m\ge 7$, there is a splitting
$\gr^{1}_{C}(\omega {^*}) \simeq 
\OOO_{C}\oplus \OOO_{C}$ or $\OOO_{C}\oplus \OOO_{C}(-1)$ by \eqref{(2.10.4)} and 
$\ov{s}(P) \neq 0 \in \gr^{1}_{C}(\omega {^*})\otimes {\CC}(P)$ 
whence $\ov{s}$ is nowhere vanishing. 
In case $m=5$, there is a splitting 
$\gr^{1}_{C}(\omega {^*}) \simeq 
\OOO_{C}\oplus \OOO_{C}(1)$ and 
\[
\ov{s}(P)=\left(\lambda (0)\cdot y_{4}+
\mu (0)\cdot (y^{m-2}_{1} - y^{2}_{2})\right)/
\Omega \in \gr^{1}_{C}(\omega {^*})\otimes {\CC}(P)
\]
is a general element because 
$\lambda (0)$ and $\mu (0)$ can be chosen arbitrary
by Corollary \xref{surj_on_omega} and by \eqref{eq-IC-omega}.
Indeed, note that
$y_4dy_2\wedge dy_4$
and $y_2^2dy_2\wedge dy_4$ are linearly independent modulo $\wedge^2\OOO_D(dy_2^m,\,
dy_2y_4^2,\, dy_2^{\frac{m+1}{2}}y_4,\, dy_4^m)$. 
Thus $\ov{s}$ is nowhere vanishing and the claim is 
proved. 
We study $E_{X}=\{s=0\} \in |-K_{X}|$. Since $\ov{s}$ is a nowhere 
vanishing section of $\gr^{1}_{C}(\omega {^*}) \simeq \omega {^*}\totimes\gr^{1}_{C}\OOO$, $E_{X}$ is smooth on $C\setminus\{P\}$. 
The canonical cover $E_{X^\sharp}$ at $P$ is defined by $y_{4}+y_{2}(\cdots)+y_{1}(\cdots)=
0$. Therefore $(E_{X},P)=(y_{1},y_{2})/\muu_{m}(2,m-2)$ has only Du Val singularities, 
whence so is $E_{Z}$ by $(K_{X}\cdot C)=0$. 

For the precise result, we 
express $(E_{X},P)=(x_{1},x_{2},x_{3}; x_{1}x_{2}=x_{3}^{m})$, where 
$x_{1}=y_{1}^{m}$, $ x_{2}=y_{2}^{m}$ 
and $ x_{3}=y_{1}y_{2}$. The curve $C$ is the image of $C^\sharp$, the locus of $ 
(t^{2},t^{m-2})$, where $C$ is the locus of $(s^{2},s^{m-2},s)$ in 
this embedding of $(E_{X},P)$, where $s=t^{m}$. 
Then it is easy to check

\par
\begin{cmp}[{see \cite[(2.10.5)]{Kollar-Mori-1992}}] 
Let $(E,P)$ be an $A_{m-1}$-singularity\textup: 
\[(E,P)=(x_{1},x_{2},x_{3}; x_{1}x_{2}=x_{3}^{m}), 
\]
and $C$ the locus of $(s^{2},s^{m-2},s)$. Then $\Delta (E, C)$ is 
as in \xref{(2.2.2)}.
\end{cmp}

Thus the proof of Theorem \xref{(2.2)} is completed in the case \type{IC}. \qed
\end{say}

\section{Case of \type{IIB}}
\label{(2.11)}

\begin{say}[\textbf{Cf. \cite[(2.11)]{Kollar-Mori-1992}}] 
Let $P \in (X,C)$ be of type \type{IIB}. Then
\[
(X,P) \simeq (y_{1},y_{2},y_{3},y_{4};\phi)/\muu_{4}(3,2,1,1;2)
\]
with $C^\sharp$ the locus of $(t^{3},t^{2},0,0)$ \cite[(A.3)]{Mori-1988}, 
where 
\[
\phi=y^{2}_{1} - y^{3}_{2}+\psi 
\]
and $\psi\in (y_{3},y_{4})$ satisfies 
$\wt\psi\equiv 2 \mod (4)$ and $\psi (0,0,y_{3},y_{4}) \notin 
(y_{3},y_{4})^{3}$. The last condition comes from the classification of 
terminal singularities \cite[(6.1)(2)]{Reid-YPG1987}. In this case, $P$ is the 
only singular point of $X$ on $C$ \cite[(B.1)]{Mori-1988}. Since $y_{3}$ and $y_{4}$ generate the defining ideal of $C^\sharp$, they form an $\ell$-free $ 
\ell$-basis of $\gr^{1}_{C}\OOO_{X}$. By residue, 
\[
\Omega=\mt{Res} \frac{dy_{1}\wedge dy_{2}\wedge dy_{3}\wedge dy_{4}}{\phi}=\frac{dy_{2}\wedge dy_{3}\wedge dy_{4}}{\partial \phi /\partial y_{1}}
\]
is an $\ell$-free $\ell$-basis of $\gr^{0}_{C}\omega_{X}$ with $\wt \Omega \equiv 1 \mod (4)$. 
\begin{lemma}[{cf. \cite[(2.11), p. 549]{Kollar-Mori-1992}}]
$i_{P}(1)=2$.
\end{lemma}
\begin{proof}
Using the parametrization 
$(t^{3},t^{2},0,0)$ of $C^\sharp$ and 
$\ell$-free $\ell$-basis $(y_{3},y_{4})$ of 
$\gr^{1}_{C}\OOO_{X}$, we see the following on 
$C^\sharp \subset X^\sharp$. 
\[
\begin{array}{ll}
\gr^{0}_{C}\omega \mid_{\tilde{C}} &=\OOO_{\tilde{C}}t^{3}\Omega \mid_{\tilde{C}} 
\\[7pt] 
&=\OOO_{\tilde{C}}tdt\wedge dy_{3}\wedge dy_{4}, 
\\[11pt] 
\wedge^{2}(\gr^{1}_{C}\OOO)\otimes \Omega^{1}_{C}\mid_{\tilde{C}} &=\OOO_{\tilde{C}}(t^{3}y_{3})\wedge (t^{3}y_{4})\otimes d(t^{4}) 
\\[7pt] 
&=\OOO_{\tilde{C}}t^{9}y_{3}\wedge y_{4}\otimes dt. 
\end{array}
\]
Thus (cf. \cite[(2.2)]{Mori-1988}) 
\[
\wedge^{2}(\gr^{1}_{C}\OOO)\otimes \Omega^{1}_{C}=t^{8}\gr^{0}_{C}\omega. 
\]
Hence $i_{P}(1)=2$ as claimed because $t^{4}$ 
is a coordinate of $C$ at $P$. 
\end{proof}

By \cite[(4.4.3)]{Mori-Prokhorov-2008} we have
$\deg \gr^{0}_{C}\omega=-1$. Then using 
\cite[(3.1.2)]{Mori-Prokhorov-2008} we obtain
$\deg \gr^{1}_{C}\OOO=-1$. 
Thus we see $\gr^{0}_{C}\omega \simeq (-1+3P^\sharp)$ and $\gr^{1}_{C}\OOO \simeq (3P^\sharp)\toplus(-1+3P^\sharp)$ with $ 
\ell$-structures using their $\ell$-free $\ell$-bases at $P$ above. Let $D=
\{y_{2}=0\}/\muu_{4}$. Then $D \in |-2K_{X}|$ by $(D\cdot C)=1/2$. By 
\[
\ql_{C}(\gr^{0}_{C}(\omega {^*}))=\ql_{C}(\omega {^*})=-\ql_{C}(\omega)=P^\sharp, 
\]
one has 
\[
\deg(\gr^{0}_{C}(\omega {^*}))=\TL(P^\sharp)=-U(-1)=-1 
\]
because 
\[
U(x)=\min\{ z \in \ZZ \mid 4z-x \in 2\ZZ_{+}+3\ZZ_{+}\}
\]
\cite[(8.9.1)(iii)]{Mori-1988}. Thus $\gr^{0}_{C}(\omega {^*}) \simeq \OOO_{C}(-1)$ and a general 
section $s \in H^{0}(\OOO_{X}(-K_{X}))$ vanishes along $C$, i.e. $s \in H^{0}(F^{1}_{C}(\omega {^*}))$. 
Hence $s=(\lambda \cdot y_{3}+\mu \cdot y_{4})/\Omega$ for some $\lambda $ and $\mu \in \OOO_{X}$. We see that 
$\lambda (0)$ and $\mu (0) \in {\CC}$ are general by Corollary \xref{surj_on_omega} (cf. \cite[(2.5)]{Kollar-Mori-1992}).
Indeed, in Corollary \xref{surj_on_omega} with 
$n=2$, we have $\omega_D\simeq\OOO_D(-K_X)$ with 
\[
\left.\frac{y_jdy_3\wedge dy_4}{y_1+(\cdots)}\right|_{D^{\sharp}}
\stackrel{\mt{Res}}{\longleftrightarrow}
\frac{y_j}{\Omega},\qquad j=3,\ 4.
\]
In view of
\[
g^*\omega_Z \subset 
\wedge^2\OOO_D
\bigl(dy_1^4, \ dy_3^4,\ dy_4^4, 
\ dy_1 y_3, \ dy_1 y_4, 
\ dy_3^3 y_4, \ dy_3^2 y_4^2, \ d y_3 y_4^3)
\]
\noindent
one easily sees
that 
$g^*\omega_Z\subset \OOO_{D^{\sharp}}(y_1,y_3,y_4)^2/\Omega$, 
and $y_3/\Omega$ and $y_4/\Omega$ are independent 
$\mod g^*\omega_Z$.

We study $E_{X}=\{s=0\} \in 
|-K_{X}|$. We see that $s$ induces a section $\ov{s}$ of 
\[
\gr^{1}_{C}(\omega {^*}) \simeq (\gr^{0}_{C}\omega)^{\totimes(-1)}
\totimes\gr^{1}_{C}\OOO \simeq (0) \toplus (1)
\]
such that $\ov{s}(P)$ is general in 
$\gr^{1}_{C}(\omega {^*})\otimes {\CC}(P)$. 
Thus $\ov{s}$ is nowhere 
vanishing, whence $E_{X}\supset C$ 
and $E_{X}$ is smooth on $C\setminus \{P\}$. Eliminating 
$y_{4}$, we see $(E_{X},P) \simeq (y_{1},y_{2},y_{3}; \ov{\phi})/\muu_{4}(3,2,1)$ with $C$ the locus of $ 
(t^{3},t^{2},0)$, where 
\[
\ov{\phi}=(y^{2}_{1} - y^{3}_{2})+y_{3}(cy_{3}+\cdots) \in {\CC}\{y_{1},y_{2},y_{3}\}
\]
for some $c \in {\CC}{^*}$ by independence of $\lambda (0)$ and $\mu (0)$. We claim that 
we may take 
\begin{equation}
\label{(2.11.1)}
\ov{\phi}=y^{2}_{1} - y^{3}_{2}+y^{2}_{3} 
\end{equation} 
modulo multiplication by units and $\muu_{m}$-automorphisms fixing $C$. First 
by Weierstrass preparation Theorem, we may assume $\ov{\phi}=y^{2}_{1}+
\alpha(y_{2},y_{3})y_{1}+\beta(y_{2},y_{3})$ with $\wt \alpha \equiv 3$ and $\wt \beta \equiv 2 \mod (4)$. 
Since $\ov{\phi}(t^{3},t^{2},0)=0$, we see $\alpha \equiv 0$ and 
$\beta \equiv y^{3}_{2} \mod (y_{3})$. Hence 
we may assume $\alpha=0$, after replacing $y_{1}$ 
by $y_{1}-\alpha/2$. Since 
$\wt ((\beta-y^{3}_{2})/y_{3}) \equiv 1$ 
and $\wt y_{2} \equiv 2 \mod (4)$, we see $\beta \equiv y^{3}_{2} \mod (y^{2}_{3})$. 
Thus \eqref{(2.11.1)} holds by 
$c \in {\CC}{^*}$. Then it is easy to check (cf. 
\cite[(4.10)]{Reid-YPG1987}) 

\begin{cmp}[see {\cite[(2.11.2)]{Kollar-Mori-1992}}] 
Let 
\[
(E,P)=(y_{1},y_{2},y_{3}; y^{2}_{3} - y^{3}_{2}+y^{2}_{3})/\muu_{4}(3,2,1;2) 
\]
and $C \subset E$ the locus of $(t^{3},t^{2},0)$. Then $(E,P)$ is $D_{5}$ and 
$\Delta (E, C)$ is as in \xref{2.2.2')}.
\end{cmp} 

Thus the proof of Theorem \xref{(2.2)} is completed in Case \type{IIB}. \qed
\end{say}

\section{Case of \type{IA}$+$\type{IA}$+$\type{III}}
\label{(2.12)}

\begin{say}[\textbf{Cf. \cite[(2.12)]{Kollar-Mori-1992}}] 
The configuration of singular points on
$(X,C)$ is the following: 
a \type{IA} point $P$ of odd index $m\ge 3$ 
and a \type{IA} point $Q$ of index $2$ and a \type{III} point $R$ \cite[(9.1)]{Mori-Prokhorov-2008}.
We know that 
$m$ is odd \cite[(9.1)]{Mori-Prokhorov-2008},
$\siz_{P}=1$ \cite[(8.5)]{Mori-Prokhorov-2008}, 
$i_{P}(1)=i_{Q}(1)=i_{R}(1)=1$ \cite[(9.2.1)]{Mori-Prokhorov-2008},
and hence
$\gr^{1}_{C}\OOO \simeq \OOO(-1)\oplus \OOO(-1)$
from the formula
\cite[(2.3.4)]{Mori-1988}. 
It follows from \cite[(3.1.1), (9.2.1), (2.8)]{Mori-Prokhorov-2008} 
that
\[
w_{P}(0)=1+(K_X\cdot C)-w_Q(0)=(m-1)/2m
\] 
(because $w_Q(0)\in \frac 12\ZZ$).
We start with the set-up. 
\end{say}

\begin{lemma}[{\cite[(2.12.1)]{Kollar-Mori-1992}}]
\label{(2.12.1)} 
We can express 
\[
\begin{array}{ll}
(X,P)&=(y_{1},y_{2},y_{3},y_{4}; \alpha)/\muu_{m}(1,\frac{m+1}{2},-1,0;0)\supset (C,P)=y_{1}\text{\rm -axis}/\muu_{m},
\\[10pt] 
(X,Q)&=(z_{1},z_{2},z_{3},z_{4}; \beta)/\muu_{2}(1,1,1,0;0)\supset (C,Q)=z_{1}\text{\rm -axis}/\muu_{2}, 
\\[10pt] 
(X,R)&=(w_{1},w_{2},w_{3},w_{4}; \gamma)\supset (C,R)=w_{1}\text{\rm -axis}, 
\end{array}
\] 
using equations $\alpha$, $\beta$ and $\gamma $ such that $\alpha \equiv y_{1}y_{3} \mod 
(y_{2},y_{3})^{2}+(y_{4})$, $\beta \equiv z_{1}z_{3} \mod (z_{2},z_{3})^{2}+(z_{4})$ and $\gamma \equiv w_{1}w_{3} \mod 
(w_{2},w_{3},w_{4})^{2}$. 
\end{lemma}

\begin{proof}
Express $(X,P)=(y_{1},y_{2},y_{3},y_{4}; \alpha)/\muu_{m}(a_{1},a_{2},-a_{1},0; 0)$ 
so that $C^\sharp$ is the locus of $(t^{a_{1}},t^{a_{2}},0,0)$, 
where $a_{1}$ and $ a_{2}$ are 
positive integers such that 
$\gcd(a_{1}a_{2},m)=1$. Since $w_{P}(0)=(m-1)/2m$, it 
holds that $a_{2}=(m+1)/2$ \cite[(4.9)(i)]{Mori-1988}. 
By $\siz_{P}=1=U(a_{1}a_{2})$, we 
have $a_{1}a_{2}\le m $ and $ a_{1}=1$. We need only to replace $y_{2}$ by 
$y_{2}-y^{(m+1)/2}_{1}$ to get the assertion for $(X,P)$. 
We can choose $\alpha$ so that
$\alpha \equiv y_{1}y_{3}$ 
because $P$ is a $cA$ point \cite[(B.1)(g)]{Mori-1988}. The rest is 
similar except for $\beta \equiv z_{1}z_{3}$ and 
$\gamma \equiv w_{1}w_{3}$ which follow from 
$i_{Q}(1)=1$ and $i_{R}(1)=1$ and \cite[(2.16)]{Mori-1988}. 
\end{proof}

We will improve the set-up in two steps.

\begin{lemma}[{\cite[(2.12.2)]{Kollar-Mori-1992}}]
\label{(2.12.2)}
The point $P$ is ordinary, that is, 
\[
(X,P)=(y_{1},y_{2},y_{3})/\muu_{m}(1,(m+1)/2,-1) 
\supset (C,P)=y_{1}\text{\rm -axis}/\muu_{m}. 
\]
\end{lemma}

\begin{proof}
Suppose that $P$ is not ordinary. We will derive a 
contradiction. By our hypothesis, 
we may assume 
$\alpha \equiv y_{1}y_{3} \mod 
(y_{2},y_{3},y_{4})^{2}$. 
Apply L-deformation at $Q$ \cite[(2.9.1)]{Kollar-Mori-1992},
see also Remark \ref{rem-def}.
If a general member of the corresponding family 
is an extremal neighborhood, the assertion follows from
\cite[(2.12.2)]{Kollar-Mori-1992}.
Thus we may assume that $Q$ is 
ordinary and hence $\beta=z_{4}+z_{1}z_{3}$ in our 
$\QQ$-conic bundle case. Hence $\{y_{2},y_{4}\}$ and 
$\{z_{2},z_{3}\}$ are the $\ell$-free $\ell$-bases 
of $\gr^{1}_{C}\OOO$ at $P$ and $Q$, respectively. By 
\cite[(2.12.1)]{Kollar-Mori-1992}, we see 
\[\textstyle
\gr^{0}_{C}\omega \simeq\left(-1+\frac{m-1}{2}P^\sharp+Q^\sharp\right)
\]
and
\[\textstyle 
\gr^{0}_{C}(\omega {^*}) \simeq\left(-1+\frac{m+1}{2}P^\sharp 
+Q^\sharp\right).
\]
Thus 
$H^{0}(\omega {^*})=H^{0}(F^{1}_{C}(\omega {^*}))$. Let $D=\{y_{1}+h(y_2,y_3,y_4)=0\}/\muu_{m}$ with general $h$ such that $\wt\ h=\wt\ y_1$. 
Then $D \in |-2K_{X}|$ by $(D\cdot C)=1/m$. 

\begin{ppar}
\label{ppar-53}
Now apply Corollary \xref{surj_on_omega} with 
$n=2$. We obtain $\omega_D\simeq\OOO_D(-K_X)$
which gives the correspondence of $\ell$-bases
\[
\left.\frac{d y_2\wedge dy_3}{\partial \alpha/\partial y_4}\right|_{D^{\sharp}}
=
\left.\frac{d y_4\wedge dy_2}{\partial \alpha/\partial y_3}\right|_{D^{\sharp}}
=
\left.\frac{d y_3\wedge dy_4}{\partial \alpha/\partial y_2}\right|_{D^{\sharp}}
\stackrel{\mt{Res}}{\longleftrightarrow}
\frac{(\mt{unit})}{\Omega},
\]
and 
\[
g^*\omega_Z \subset \wedge^2
\OOO_D\bigl(dy_2^m,\ dy_3^m,\ dy_4,\ 
dy_2y_3^{(m+1)/2},\ dy_2^2y_3\bigr).
\]
Hence,
\[
\begin{array}{lc}
g^*\omega_Z \subset& 
\sum_{i,j=2,3,4} \OOO_{D^\sharp}(y_2,y_3)^2 dy_i\wedge dy_j
\\\
 &{\updownarrow}\lefteqn{\scriptstyle\mt{Res}}{}
\\
&\sum_{k=2,3,4}
\OOO_{D^\sharp}(y_2,y_3)^2
(\partial \alpha/\partial y_k)\Omega^{-1}.
\end{array}
\]
So we have the lifting modulo $\OOO_{D^{\sharp}}(y_2,y_3)^2(\partial\alpha/\partial y_2,\partial\alpha/\partial y_3,\partial\alpha/\partial y_4)$.
Therefore
there exists $s \in H^{0}(F^{1}_{C}(\omega {^*}))$ 
inducing $(y_{2}+(y_{1}+h)\OOO_{X})/\Omega \in \OOO_{D}(-K_{X})$, 
where
\[
\Omega=\left.\frac{dy_1\wedge dy_2 \wedge dy_3}
{\partial \alpha/\partial y_4}\right|_{D^{\sharp}}.
\]
\end{ppar}

\begin{ppar}
Thus $s$ induces a global section 
$\ov{s}$ of $\gr^{1}_{C}(\omega {^*}) \simeq \gr^{1}_{C}\OOO \totimes \gr^{0}_{C}(\omega {^*})$ which is a part of $\ell$-free $\ell$-basis 
at $P$. Hence there is an exact sequence 
\[
0 \to \gr^{0}_{C}\omega \to \gr^{1}_{C}\OOO \to \gr^{1}_{C}\OOO/\gr^{0}_{C}\omega \to 0. 
\]
It is split because $\gr^{0}_{C}\omega \simeq \OOO(-1)$ and $\gr^{1}_{C}\OOO \simeq \OOO(-1)\oplus \OOO(-1)$. Then it 
is $\ell$-split at $Q$ because $\ell$-bases of 
$\gr^{0}_{C}\omega$ and $\gr^{1}_{C}\OOO$ have 
all weights $\equiv 1 \mod (2)$. Hence
$\gr^{1}_{C}\OOO/\gr^{0}_{C}\omega$ is an 
$\ell$-invertible sheaf such that 
\[
\ql_{C}(\gr^{1}_{C}\OOO/\gr^{0}_{C}\omega)=\ql_{C}(\gr^{1}_{C}\OOO) - \ql_{C}(\gr^{0}_{C}\omega)=-1+Q^\sharp. 
\]
Applying \cite[(2.6)]{Kollar-Mori-1992} to 
$\ql_{C}(\gr^{0}_{C}\omega) - 
\ql_{C}(\gr^{1}_{C}\OOO/\gr^{0}_{C}\omega)
=\frac{m-1}{2}P^\sharp$, we obtain
an $\ell$-splitting 
\begin{equation}
\label{(2.12.2.1)}
\textstyle
\gr^{1}_{C}\OOO \simeq 
\left(-1+\frac{m-1}{2}P^\sharp+Q^\sharp\right) 
\toplus (-1+Q^\sharp). 
\end{equation}
We may further assume that 
$y_{2}$, $z_{2}$ and $w_{2}$ (resp. $y_{4}$, $z_{3}$ 
and $w_{4}$) are the $\ell$-free $\ell$-bases of 
$\left(-1+\frac{m-1}{2}P^\sharp+Q^\sharp\right)$ 
(resp. $(-1+Q^\sharp)$) 
at $P$, $Q$ and $R$, by making coordinates changes to the ones in 
\cite[(2.12.1)]{Kollar-Mori-1992}. 
\end{ppar}

Let $J$ be the $C$-laminal ideal of width 2 such that 
$J/F^{2}_{C}\OOO=\left(-1+\frac{m-1}{2}P^\sharp+Q^\sharp\right)$. 
Then $\{y_{2},y_{3},y^{2}_{4}\}$ form an $\ell$-basis of $J$ at $P$. 
By replacing $y_{3}$ by an element of the form 
$y_{3}+y^{2}_{4}(\cdots)$ if 
necessary, we may assume $\alpha \equiv y_{1}y_{3}+cy^{2}_{4} \mod J^\sharp I_{C^\sharp}$ for some 
$c \in {\CC}$. If $c \neq0$ then $I\supset J$ 
is $(1,2,2)$-monomializable at $P$ 
(see \cite[(8.9-8.10)]{Mori-1988}). 
If $c=0$ we may still assume that $I\supset J$ is $(1,2,2)$-monomializable at $P$ by 
deformation arguments \cite[(2.9.2)]{Kollar-Mori-1992}
(see also Remark \ref{rem-def})
because Lemma \ref{(2.12.2)} would follow from 
\cite[(2.12.2)]{Kollar-Mori-1992} 
if our $(X,C)$ deformed to 
an extremal neighborhood. 
In the same way, we may 
assume that $I\supset J$ is $(1,2,2)$-monomializable at $R$. At the 
ordinary point $Q$, $I\supset J$ is $(1,2)$-monomializable. 
Thus there are $\ell$-isomorphisms 
\[
\begin{array}{cl}
\gr^{1}(\OOO,J) &\simeq (-1+Q^\sharp), 
\\[10pt]
\gr^{2,0}(\OOO,J) &\simeq 
\left(-1+\frac{m-1}{2}P^\sharp+Q^\sharp\right), \\[10pt] 
\gr^{2,1}(\OOO,J) &\simeq \gr^{1}(\OOO,J)^{\totimes2}\totimes(1+P^\sharp) \simeq (P^\sharp), 
\\[10pt] 
\gr^{3,0}(\OOO,J) &\simeq \gr^{2,0}(\OOO,J)\totimes\gr^{1}(\OOO,J) \simeq\left(-1+\frac{m-1}{2}P^\sharp\right), 
\\[10pt] 
\gr^{3,1}(\OOO,J) &\simeq \gr^{2,1}(\OOO,J)\totimes\gr^{1}(\OOO,J) \simeq (-1+P^\sharp+Q^\sharp) 
\end{array}
\]
(cf. \cite[(8.6), (8.12)]{Mori-1988}) and the following:
\[
\begin{array}{cl}
\gr^{1}(\omega,J) &\simeq \gr^{1}(\OOO,J) \totimes gr^0_C \omega \simeq (-1+\frac{m-1}{2}P^{\sharp}), 
\\[10pt]
\gr^{2,0}(\omega,J) &\simeq 
\gr^{2,0}(\OOO,J) \totimes gr^0_C \omega \simeq \left(-1+(m-1)P^\sharp\right), \\[10pt] 
\gr^{2,1}(\omega,J) &\simeq \gr^{2,1}(\OOO,J) \totimes gr^0_C \omega \simeq 
\left(-1+\frac{m+1}{2}P^\sharp+Q^\sharp\right), 
\\[10pt] 
\gr^{3,0}(\omega,J) &\simeq \gr^{3,0}(\OOO,J)\totimes\gr^0_C\omega \simeq\left(-2+(m-1)P^\sharp+Q^\sharp\right), 
\\[10pt] 
\gr^{3,1}(\omega,J) &\simeq \gr^{3,1}(\OOO,J)\totimes\gr^0_C\omega \simeq \left(-1+\frac{m+1}{2}P^\sharp\right) ,
\end{array}
\]
by $gr^0_C \omega \simeq (-1+\frac{m-1}{2}P^{\sharp}+Q^{\sharp})$. 

Hence there are an 
$\ell$-isomorphism and $\ell$-exact sequences 
\[
\textstyle
\gr^{1}(\omega,J) \simeq\left(-1+\frac{m-1}{2}P^\sharp\right), 
\]
\[
\textstyle
0 \to\left(-1+\frac{m+1}{2}P^\sharp+Q^\sharp\right) \to \gr^{2}(\omega,J) \to (-1+(m-1)P^\sharp) \to 0, 
\]
\[
\textstyle
0 \to\left(-1+\frac{m+1}{2}P^\sharp\right) 
\to \gr^{3}(\omega,J) \to (-2+(m-1)P^\sharp+Q^\sharp) \to 0,
\]
by \cite[(8.6)]{Mori-1988}.
Now from the exact sequences 
\begin{equation}
\label{eq-omegaJF}
0\to \gr^{n}(\omega,J) \to \omega_X/F^{n+1}(\omega,J)
\to \omega_X/F^{n}(\omega,J)\to 0
\end{equation}
we obtain $H^{1}(\omega /F^{4}(\omega,J)) \neq 0$ which is a contradiction. Then it follows from 
\cite[(4.4)]{Mori-Prokhorov-2008} that
$V:=\Spec_X \OOO_X/F^4(\OOO,J)\supset f^{-1}(o)$.
Hence, $2=(-K_X\cdot f^{-1}(o))\le (-K_X\cdot V)$.
On the other hand, 
near a general point $S\in C$,
for a suitable choice of coordinates $(x,y,z)$ in $(X,S)$,
we may assume that 
$F^1(\OOO,J)=I_C=(x,y)$,
$F^2(\OOO,J)=J=(x,y^2)$,
$F^3(\OOO,J)=I_CJ+J=(x^2,xy,y^3)$,
$F^4(\OOO,J)=J^2=(x^2,xy^2,y^4)$.
Hence, 
\[\textstyle
2\le (-K_X\cdot V)
=\frac{1}{2m}\mt{length}_S\CC\{x,y\}/(x^2,xy^2,y^4)
=\frac{6}{2m},
\]
which is a contradiction. Lemma \xref{(2.12.2)} is 
proved. 
\end{proof}

\par
\begin{lemma}[{\cite[(2.12.3)]{Kollar-Mori-1992}}]
\label{(2.12.3)} 
The point $Q$ is ordinary, that is, 
\[
(X,Q)=(z_{1},z_{2},z_{3})/\muu_{2}(1,1,1)\supset (C,P)=z_{1}\text{\rm -axis}/\muu_{2}. 
\]
\end{lemma}

\begin{proof}
Assuming that $Q$ is not ordinary whence $\beta \equiv z_{1}z_{3} 
\mod (z_{2},z_{3},z_{4})^{2}$, we will derive a 
contradiction. As in the proof of 
Lemma \xref{(2.12.2)}, there is a split exact 
sequence 
\[
0 \to \gr^{0}_{C}\omega \to \gr^{1}_{C}\OOO \to (\gr^{1}_{C}\OOO/\gr^{0}_{C}\omega) \to 0
\]
which is $\ell$-split at $P$. 
Since $\ell$-free $\ell$-bases of $\gr^{0}_{C}\omega$ 
(resp. $\gr^{1}_{C}\OOO$) at $Q$ have weights $1$
(resp. $0$, $1$) $\mod (2)$, 
the above sequence is also $\ell$-split at $Q$. 
Thus there are $\ell$-exact sequences 
\[
\textstyle
0 \to\left(-1+\frac{m-1}{2}P^\sharp+Q^\sharp\right) 
\to \gr^{1}_{C}\OOO \to (-1+P^\sharp) \to 0, 
\]
\[
\textstyle
0 \to (-1+(m-1)P^\sharp) \to \gr^{1}_{C}\omega \to 
\left(-2+\frac{m+1}{2}P^\sharp+Q^\sharp\right) \to 0.
\]
Similarly
to the argument at the end of Lemma \xref{(2.12.2)}
$H^{1}(\omega /F^{2}_{C}\omega) \neq0$
and one
has a contradiction by $2\le 3/(2m)$.
Lemma \xref{(2.12.3)} is proved. 
\end{proof}

\begin{say}[{\cite[(2.12.4)]{Kollar-Mori-1992}}] 
\label{(2.12.4)}
As in the argument for Lemma \xref{(2.12.2)}, there is 
an $\ell$-isomorphism 
\[
\textstyle
\gr^{1}_{C}\OOO \simeq 
\left(-1+\frac{m-1}{2}P^\sharp+Q^\sharp\right) 
\toplus 
\left(-1+P^\sharp+Q^\sharp\right). 
\]
Let $J$ be the $C$-laminal ideal such that 
$J/F^{2}_{C}\OOO=
\left(-1+\frac{m-1}{2}P^\sharp+Q^\sharp\right)$. 
After an (equivariant) change of coordinates if necessary, we may 
assume that $(y_{2},z_{2},w_{2})$ (resp. $(y_{3},z_{3},w_{4})$) are $\ell$-free $\ell$-bases of 
$\left(-1+\frac{m-1}{2}P^\sharp+Q^\sharp\right)$ 
(resp. $(-1+P^\sharp+Q^\sharp)$), whence $J=(w_{2},w_{3},w^{2}_{4})$ 
at $R$. Replacing $w_{3}$ by an element $\equiv w_{3} \mod (w_{2},w_{4})^{2}$ if 
necessary, we may further assume 
\[
\gamma \equiv w_{1}w_{3}+c_{1}w^{2}_{4}+c_{2}w_{4}w_{2}+c_{3}w^{2}_{2} \mod (w_{3},w^{2}_{2},w_{2}w_{4},w^{2}_{4})\cdot I_{C}
\]
for some $c_{1}, c_{2}, c_{3} \in {\CC}$. We note $\gamma \equiv w_{1}w_{3}+c_{1}w^{2}_{4} \mod J\cdot I_{C}$. 
\end{say}

\begin{lemma}[{\cite[2.12.5]{Kollar-Mori-1992}}]
\label{(2.12.5)} 
A general member $E_{X}$ of $|-K_{X}|$ has 
singularities $A_{m-1}, A_{1}$ and $A_{n}$ at $P$, $Q$ and $R$, respectively and 
is smooth elsewhere, and $\Delta (E_{X}, C)$ is 
\[
\begin{array}{ll}
& \circ \\ 
& \mid \\ 
\underbrace{\circ -\cdots - \circ}_{m-1} -&\bullet -\underbrace{\circ -\cdots - \circ}_{n}, 
\end{array}
\]
where $n$ is some integer $\ge 1$. We have 
$n=1$ if $c_{1}\neq0$ when $m\ge 5$ or if $(c_{1},c_{2},c_{3}) \neq (0,0,0)$ when $m=3$. 
\end{lemma}

\begin{proof}
There is an $\ell$-isomorphism 
$\gr^{1}_{C}(\omega {^*}) \simeq (0) \toplus (-1+\frac{m+3}{2}P^\sharp)$. 
Let $D=\{y_{1}+h=0\}/\muu_{m} \in |-2K_{X}|$ 
as before. We treat the case $m\ge 5$. 
Then 
$\gr^{1}_{C}\omega {^*} \simeq \OOO_{C} \oplus \OOO_{C}(-1)$ and $H^{0}(\OOO(-K_{X}))=H^{0}(\gr^{2}_{C}(\omega {^*},J))$. 
As in \ref{ppar-53} one has $H^0(\OOO_X(-K_X)) \twoheadrightarrow \omega_D \mod (y_2,y_3)^2\omega_{D^{\sharp}}$.
So a 
general section $s \in H^{0}(\OOO(-K_{X}))$ induces $(y_{2}+\cdots)/\Omega$,
up to some 
units whence induces a non-zero global section $\ov{s}$ of $\gr^{1}_{C}\omega {^*}$. Hence 
$\ov{s}$ is nowhere vanishing 
and the defining equations of $E_{X}=\{s=0\}$ 
are $y_{2}$, $z_{2}$ and $w_{2} \mod F^{2}_{C}\OOO$ up to units at $P$, $Q$ and $R$, respectively. Then $ 
E_{X}$ is smooth outside of $P$, $Q$ and $R$, $(E_{X},P) \simeq (y_{1},y_{3})/\muu_{m}(1,-1)$, $ 
(E_{X},Q) \simeq (z_{1},z_{3})/\muu_{2}(1,1)$ and $(E_{X},R) \simeq (w_{1},w_{3},w_{4};\ov{\gamma})$, where 
$\ov{\gamma}(w_{1},w_{3},w_{4}) 
\equiv w_{1}w_{3}+c_{1}w^{2}_{4} 
\mod (w_{3},w^{2}_{4})(w_{3},w_{4})$. We are done in case $m\ge 5$. In case $m=3$, we can see that $\gr^{1}_{C}\omega {^*} \simeq (0)\toplus(0)$ and 
$H^{0}(\OOO(-K_{X})) \twoheadrightarrow 
H^{0}(\gr^{1}_{C}\omega {^*})$ and we get a similar assertion on $E_{X}$ except 
that $\gamma \equiv w_{1}w_{3}+(c_{3}t^{2}+c_{2}t+c_{1})w^{2}_{4}$ for some general $t \in {\CC}$. Thus 
we are done in case $m=3$. 
\end{proof}

\begin{lemma}[{\cite[(2.12.6)]{Kollar-Mori-1992}}]
\label{(2.12.6)} 
If $m\ge 5$, then $c_{1} \neq0$ and $n=1$ in Lemma \xref{(2.12.5)}. Thus the case \xref{(2.2.3')} holds when $m\ge 5$.
\end{lemma} 

\begin{proof}
Assume that $m\ge 5$ and $c_{1}=0$. By 
$w_{1}w_{3} \in J\cdot I_{C}$, we 
have $w_{3} \in F^{3}(\OOO,J)$ 
and $\gr^{2}(\OOO,J)=\OOO_{C}w_{2}\oplus \OOO_{C}w^{2}_{4}$ at $R$. Thus there are $\ell$-isomorphisms (cf. the proof of Lemma \xref{(2.12.2)})
\[
\begin{array}{cl}
\gr^{1}(\OOO,J) &=(-1+P^\sharp+Q^\sharp), 
\\[10pt] 
\gr^{2,0}(\OOO,J) &=
\left(-1+\frac{m-1}{2}P^\sharp+Q^\sharp\right), \\[10pt] 
\gr^{2,1}(\OOO,J) &=\gr^{1}(\OOO,J)^{\totimes2} \simeq (-1+2P^\sharp). 
\end{array}
\]
Therefore,
\[
\textstyle
\gr^{1}(\omega,J) \simeq 
\left(-1+\frac{m+1}{2}P^\sharp\right), 
\]
\[
\textstyle
0\to\left(-2+\frac{m+3}{2}P^\sharp+Q^\sharp\right) 
\to \gr^{2}(\omega,J)\to (-1+(m-1)P^\sharp) \to 0. 
\]
>From the exact sequence 
\[
0\to \gr^{2}(\omega,J) \to \omega_X/F^{3}(\omega,J)
\to \omega_X/F^{2}(\omega,J)\to 0
\]
we obtain $H^{1}(\omega /F^{3}(\omega,J)) \neq0$.
Hence $2\le 4/(2m)$ \cite[(4.4)]{Mori-Prokhorov-2008}, 
a contradiction.
\end{proof}

\begin{lemma}[{cf. 
\cite[(2.12.7)]{Kollar-Mori-1992}}]
\label{(2.12.7)} 
Assume $m=3$. 
If $(c_{1},c_{2},c_{3}) \neq 0$ \textup(resp. $=0$\textup)
then $n=1$ \textup(resp. $=2$\textup) in Lemma \xref{(2.12.5)}. 
Thus the case \xref{(2.2.3')} holds when $m =3$.
\end{lemma} 

\begin{proof}
Assume that $(c_{1},c_{2},c_{3})=0$. Then $w_{3} \in F^{3}_{C}\OOO$. Changing $w_{1}$ and $w_{3}$, 
we may further assume 
$\gamma=w_{1}w_{3}+\delta (w_{2},w_{4})$, 
where $\delta$ 
is a power series in $w_{2}$ and $w_{4}$ of order $d\ge 3$. 
Then $d=3$ because 
$2\cdot \ldeg_{C}(-1+P^\sharp+
Q^\sharp)+1/d\ge 0$ 
\cite[(2.12.8)]{Kollar-Mori-1992}.
In the proof of Lemma \xref{(2.12.5)}, it is easy
to see $n=d-1$ from $\gr_C^1(\OOO)=(-1+P^{\sharp}+Q^{\sharp})
\toplus (-1+P^{\sharp}+Q^{\sharp})$.
\end{proof} 

Thus we end up with the case \xref{(2.2.3')} for
\type{IA}$+$\type{IA}$+$\type{III}, and
the proof of Theorem \xref{(2.2)} is completed 
for \type{IA}$+$\type{IA}$+$\type{III}.

\section{Case of \type{IA}$+$\type{IA}}
\label{(2.13)}

\begin{say}[\textbf{Cf. \cite[(2.13)]{Kollar-Mori-1992}}]
In this section, we consider the case \type{IA}$+$\type{IA}.
Note that $|-K_X|$ has a Du Val member when
both indices are $3$ or larger \cite{Mori-Prokhorov-2008}.
Thus we can assume that the singular locus of $(X,C)$ 
consists of 
a \type{IA} point $P$ of odd index $m\ge 3$ and
a \type{IA} point $Q$ of index $2$ \cite{Mori-Prokhorov-2008}.
We know that $\siz_{P}=1$, by \cite[(8.5)]{Mori-Prokhorov-2008}.
\end{say}

We start with the set-up. 
The following is very similar to Lemma \ref{(2.12.1)}.

\begin{lemma}[{cf. \cite[(2.13.1)]{Kollar-Mori-1992}}]
\label{(2.13.1)} 
We can write 
\[
\begin{array}{cl}
(X,P)&=(y_{1},y_{2},y_{3},y_{4}; \alpha)/\muu_{m}(1,
\frac{m+1}2,-1,0;0)\supset (C,P)=y_{1}\text{\rm -axis}/\muu_{m},
\\[10pt] 
(X,Q)&=(z_{1},z_{2},z_{3},z_{4}; \beta)/\muu_{2}(1,1,1,0;0) 
\supset (C,Q)=z_{1}\text{\rm -axis}/\muu_{2}, 
\end{array}
\]
using equations $\alpha$ and $\beta$ 
such that $\alpha \equiv y_{1}y_{3}\mod (y_{2},y_{3})^{2}+(y_{4})$. 
\end{lemma}

We recall 
$\ell(P)=\mt{length}_{P^\sharp}
\bigl(I{^\sharp}^{(2)}/I{^\sharp}^{2}\bigr)$, 
where $I^\sharp$ is the defining 
ideal of $C^\sharp$ in $(X^\sharp,P^\sharp)$ and $\ell (Q)$ is defined similarly. 

\begin{lemma}[{cf. \cite[(2.13.2)]{Kollar-Mori-1992}}]
\label{(2.13.2)}
Either $\ell (P) =0$ or $1$, and $i_{P}(1)=1$. 
\end{lemma} 

\begin{proof}
This follows from $\alpha \equiv y_{1}y_{3}$ and \cite[(2.16)]{Mori-1988}. 
\end{proof}

\begin{lemma}[{cf. \cite[(2.13.3)]{Kollar-Mori-1992}}]
\label{(2.13.3)}
Either 

\begin{case}[{\cite[(2.13.3.1]{Kollar-Mori-1992}}]
\label{(2.13.3.1)} $\ell (Q) =0$ or $1$ (in particular, the point $(X,Q)$ 
is of type $cA/2$), $i_{Q}(1)=1$, and 
$\gr^{1}_{C}\OOO \simeq \OOO\oplus \OOO(-1)$, or 
\end{case}

\begin{case}[{\cite[(2.13.3.2]{Kollar-Mori-1992}}]
\label{(2.13.3.2)} 
$\ell (Q)=2$, $i_{Q}(1)=2$, 
$\gr^{1}_{C}\OOO \simeq \OOO(-1)\oplus \OOO(-1)$, 
and $P$ is ordinary:
\[
\textstyle
(X,P)=(y_{1},y_{2},y_{3})/\muu_{m}\left(1,
\frac{m+1}2,-1\right) 
\supset (C,P)=y_{1}\text{-axis}/\muu_{m}.
\]
\end{case}
\end{lemma}

\begin{proof} 
The assertion on $i_{Q}(1)$ follows from the one on $\ell (Q)$ by 
$i_Q(1)=[\ell(Q)/2]+1$
\cite[(2.16)]{Mori-1988}. 
We assume $\ell (Q)\ge 2$ and denote it by $r$. Thus 
we may choose $\beta \equiv z^{r}_{1}z_{i} \mod (z_{2},z_{3},z_{4})^{2}$, where $i=3$ (resp. 4) 
if $r \equiv 1$ (resp. 0) $\mod (2)$. 
If we extend (see Remark \ref{rem-def}) the deformation 
$\beta+
tz^{r-2}_{1}z_{i}=0$ of $(X,Q)$ to a deformation 
$(X_{t}, C_{t})\ni Q_{t}$ of 
$(X, C)\ni Q$ which is trivial outside of a small neighborhood
of $Q$, then $X_{t}$ has two 
\type{IA} points and one \type{III} point on $C_{t}$ and $\beta+tz^{r-2}_{1}z_{i} =0$ is the 
equation for $(X_{t}^{\sharp},Q_{t}^{\sharp})$ 
(cf. \cite[(4.12.2)]{Mori-1988}). Hence 
$Q_{t}$ is ordinary, that is, $r=2$ by 
Lemma \xref{(2.12.3)} or \cite[(2.12)]{Kollar-Mori-1992}. 
\end{proof}

First we treat the special case \xref{(2.13.3.2)}. 

\begin{lemma}[{cf. \cite[(2.13.4)]{Kollar-Mori-1992}}]
\label{(2.13.4)}
Assume that we are in the case \xref{(2.13.3.2)}.
Then holds the case \xref{(2.2.3)}, that is, the case \type{IA}$+$\type{IA}, 
and the singular point $Q$ is of type 
$cA/2$, $cAx/2$ or $cD/2$. 
\end{lemma}

\begin{proof}
The argument is quite similar to the case \type{IA}$+$\type{IA}$+$\type{III} (Section \xref{(2.12)}). 
As in the paragraph \xref{(2.12.4)}, 
there is an $\ell$-isomorphism 
\[
\textstyle
\gr^{1}_{C}\OOO \simeq 
\left(-1+\frac{m-1}{2}P^\sharp+Q^\sharp\right) \toplus (-1+P^\sharp+Q^\sharp), 
\]
and let $J$ be the $C$-laminal ideal such that 
\[
\textstyle
J/F^{2}_{C}\OOO=\left(-1+\frac{m-1}{2}P^\sharp+
Q^\sharp\right). 
\]
We may assume that $(y_{2},z_{2})$ (resp. $(y_{3},z_{3})$) are $\ell$-free $\ell$-bases of 
$\left(-1+\frac{m-1}{2}P^\sharp+Q^\sharp\right)$ 
(resp. $(-1+P^\sharp+Q^\sharp)$), $J^\sharp=
(z_{2},z_{4},z^{2}_{3})$ and 
\[
\beta \equiv z^{2}_{1}z_{4}+c_{1}z^{2}_{3}+c_{2}z_{2}z_{3}+c_{3}z^{2}_{2} \mod (z_{4},z^{2}_{3},z_{2}z_{3},z^{2}_{2})\cdot I_{C}
\]
at $Q$ for some $c_{1}$, $c_{2}$, $c_{3} \in {\CC}$. 
We note that 
$\beta \equiv z^{2}_{1}z_{4}+c_{1}z^{2}_{3}\mod 
J^\sharp I^\sharp$. 
The following Lemma \xref{(2.13.5)} 
corresponds to Lemma \xref{(2.12.5)}. The fact 
that $(c_{1},c_{2},c_{3}) \neq (0,0,0)$ and the 
assertion on the type of $Q$ 
follows from the classification of 
terminal 3-fold singularities \cite[(6.1)]{Reid-YPG1987}. 
The assertion that 
$c_{1} \neq0$ for the case $m\ge 5$ is proved in the same way as 
Lemma \xref{(2.12.6)}. Thus 
Lemma \xref{(2.13.4)} is proved.
\end{proof}

\begin{lemma}[{\cite[(2.13.5)]{Kollar-Mori-1992}}]
\label{(2.13.5)}
Under the notation of the previous proof,
assume that $c_{1} \neq0$ when $m\ge 5$, or $(c_{1},c_{2},c_{3}) 
\neq (0,0,0)$ when $m=3$. Then for a general member $E_{X}$ of $|-K_{X}|$, $\Delta (E_{X}, C)$ is 
\[
\begin{array}{ll}
&\circ 
\\ 
&\mid 
\\ 
\underbrace{\circ -\cdots - \circ}_{m-1} - \bullet - \underbrace{\circ -\cdots - \circ}_{2k-3} -&\circ - \circ, 
\end{array}
\]
where $k (\ge 2)$ is the axial multiplicity of $(X,Q)$. 
\end{lemma}

\begin{proof}
The only difference from Lemma \xref{(2.12.5)} 
is the analysis of the 
singularity $(E_{X},Q) \simeq (z_{1},z_{3},z_{4};\ov{\beta})/\muu_{2}(1,1,0;0)$, 
where $\ov{\beta}$ satisfies $\ov{\beta} 
\equiv z^{2}_{1}z_{4}+z^{2}_{3} \mod (z_{4},z^{2}_{3})(z_{4},z_{3})$ and $\ord \ov{\beta}(0,0,z_{4})=k < \infty$. It is 
easy to see that $\ov{\beta}=z^{2}_{1}z_{4}+z^{2}_{3}+z^{k}_{4}$ modulo formal $\muu_{m}$-automorphisms 
in $(z_{1},z_{3},z_{4})$. Thus it is reduced to the following explicit 
computation (cf. \cite[(4.10)]{Reid-YPG1987}). 
\end{proof}

\begin{cmp}[{\cite[(2.13.6)]{Kollar-Mori-1992}}]
\label{(2.13.6)}
Let 
\[
(E,Q)=(z_{1},z_{3},z_{4}; z^{2}_{1}z_{4}+z^{2}_{3}+z^{k}_{4})/\muu_{2}(1,1,0;0)
\]
and $C=z_{1}$-axis$/\muu_{2}$, 
where $k\ge 2$. Then $(E,Q)$ is $D_{2k}$ and $\Delta (E, C)$ 
is 
\[
\begin{array}{ll}
&\circ \\ 
&\mid \\ 
\bullet - \underbrace{\circ -\cdots - \circ}_{2k-3} -&\circ - \circ. 
\end{array}
\]
\end{cmp} 

\begin{say}[Cf. {\cite[(2.13.7)]{Kollar-Mori-1992}}]
\label{(2.13.7)} 
In the rest of this chapter, we assume the case 
\xref{(2.13.3.1)} unless otherwise mentioned. 

We choose an $\ell$-splitting 
$\gr^{1}_{C}\OOO \simeq \LLL\toplus \MMM$ as in 
\cite[(2.8)]{Kollar-Mori-1992} (see \cite[9.1.7]{Mori-1988}) 
such that $\deg \LLL=0$ and 
$\deg \MMM=-1$, see \eqref{(2.13.3.1)}. 
Let $J$ be the $C$-laminal ideal of 
width 2 such that $J/F^{2}_{C}\OOO=\LLL$. 
For an $\ell$-invertible sheaf $F$ with an 
$\ell$-free $\ell$-basis $f$ at a point $T$ of 
index $n$, we can give an 
equivalent definition of $\qldeg(F,T) \in [0,n)$ as 
$\qldeg(F,T) \equiv - \wt f \mod (n)$. 
(This is because $(C^\sharp, P^\sharp)$ and 
$(C^\sharp, Q^\sharp)$ are smooth.) 
\end{say}

\begin{lemma}[{\cite[(2.13.8)]{Kollar-Mori-1992}}]
\label{(2.13.8)}
$\qldeg(\MMM,Q)=1$. 
\end{lemma}

\begin{proof}
We assume $\qldeg(\MMM,Q)=0$. 
Then $\MMM \simeq (-1+iP^\sharp)$ for
$i=0$, $1$ 
or $(m-1)/2$ since $y_{2}$, $y_{3}$ and $y_{4}$ 
generate
$\gr^{1}_{C}\OOO^{\sharp}$ 
at $P^{\sharp}$. 
It follows from 
$\ql_{C}(\gr^{0}_{C}\omega)=-1+\frac{m-1}{2}P^\sharp+Q^\sharp$, 
that
\[
\textstyle
\gr^{1}_{C}\omega \simeq \gr^{1}_{C}{\OOO
\totimes}\gr^{0}_{C}\omega 
\simeq \LLL\totimes\gr^{0}_{C}\omega 
\toplus 
\left(-2+\left(\frac{m-1}2+i\right)P^\sharp+Q^\sharp\right). 
\]
Since $(m-1)/2+i\le m-1 < m$, we have 
$H^{1}(\gr^{1}_{C}\omega) \neq 0$. This is a contradiction 
to $H^{1}(\omega /F^{2}_{C}\omega) =0$ because of $H^{1}(\gr^{0}_{C}\omega)=0$.
Indeed, otherwise by \cite[(4.4)]{Mori-Prokhorov-2008} we have 
$f^{-1}(o) \subset \Spec \OOO_X/F^2_C \OOO$ 
and $2\le 3/(2m)$, which is a contradiction.
\end{proof} 

\begin{remark}[{\cite[(2.13.8.1)]{Kollar-Mori-1992}}]
\label{(2.13.8.1)}
For comparison with \cite[(9)]{Mori-1988}, it might be 
worthwhile to mention\footnote{$q(-)$ in \cite[2,(13.8.1)]{Kollar-Mori-1992} was
$\qldeg(\LLL,-)$.}
\[
\begin{array}{ll}
\qldeg(\MMM,Q)=1 &\text{iff}\ \ell (Q)+\qldeg(\LLL,Q)=1, \\[10pt] 
\qldeg(\MMM,P)=\frac{m-1}2 &\text{iff}\ \ell (P)+\qldeg(\LLL,P)=1. 
\end{array}
\]
\end{remark} 

\begin{lemma}[{Corresponds to but different from \cite[(2.13.9)]{Kollar-Mori-1992}}]
\label{(2.13.9)}
$\qldeg(\MMM,P) \neq (m-1)/2$
\end{lemma}

\begin{proof}
We assume $\qldeg(\MMM,P)=(m-1)/2$ to the contrary.
There is 
an $\ell$-isomorphism $\MMM \simeq \gr^{0}_{C}\omega$. We may assume that $y_{2}$ is an $\ell$-free $\ell$-basis of $\MMM$ at $P$. Let $D=\{y_{1}=0\}/\muu_{m}$. 
It is easy to see $D \in |-2K_{X}|$ by $(D\cdot C)=
1/m$. By 
$H^{0}(\OOO(-K_{X}))=H^{0}(F^{1}_{C}(\omega {^*}))$, its general section $s$ 
induces a section $\ov{s}$ of 
$\gr^{1}_{C}\omega {^*} 
\simeq \LLL\totimes(\gr^{0}_{C}\omega)^{\totimes(-1)}\toplus (0)$. 
Similar to arguments in \ref{ppar-53} one can see 
that the projection 
of $\ov{s}$ to $(0)$ is non-zero
because $y_{2}/\Omega$ is an $\ell$-free $\ell$-basis of 
$(0)$ at $P$ and $s$ induces an element of the form $y_{2}/\Omega+\cdots$ up to 
units, where $\Omega$ is an $\ell$-free $\ell$-basis of $\gr^{0}_{C}\omega$ at $P$. Thus $\ov{s}$ is 
nowhere vanishing, whence $E_{X}=\{s=0\}$ is smooth outside of 
$P$ and $Q$. 
The analysis of $(E_{X},P)$ and $(E_{X},Q)$ is the same as \cite[(9.9.3)]{Mori-1988}. Hence $(E_Z,o')$ has a configuration:
\[
\circ -\cdots - \circ - \bullet - \circ -\cdots - \circ.
\]
The difference from \cite{Kollar-Mori-1992} is that this implies
that $X$ is of index $2$ by \cite[(11.2)]{Mori-Prokhorov-2008} 
in our $\QQ$-conic bundle germ case where
the base is smooth. Since the index $m$ of $P$ is odd 
and $>1$, this is a contradiction and we are done.
\end{proof} 

\begin{lemma}[{\cite[(2.13.10)]{Kollar-Mori-1992}}]
\label{(2.13.10)}
The point $P$ is 
ordinary and $m\ge 5$. After changing 
coordinates, we may assume
\[
\begin{array}{ll}
(X,P)&=(y_{1},y_{2},y_{3})/\muu_{m}(1,(m+1)/2,-1)\supset (C,P)=y_{1}\text{-axis}/\muu_{m},\\[10pt] 
(X,Q)&=(z_{1},z_{2},z_{3},z_{4};\beta)/\muu_{2}(1,1,1,0;0)\supset (C,Q)=z_{1}\text{-axis}/\muu_{2}; 
\end{array}
\]
$y_{2}$ and $y_{3}$ are $\ell$-free $\ell$-bases of $\LLL$ and $\MMM$ at $P$ respectively; $z_{3}$ 
\textup(resp. $z_{4}$\textup) 
and $z_{2}$ are $\ell$-free $\ell$-bases of 
$\LLL$ and $\MMM$ at $Q$ 
respectively, 
\begin{equation}
\label{eq-IAIA-LM}
\begin{array}{ll}
\LLL&=(\frac{m-1}{2}P^\sharp+Q^\sharp)\ 
(\text{\rm resp.}\ \LLL = (\frac{m-1}{2}P^\sharp)), \\[10pt] 
\MMM&=(-1+P^\sharp+Q^\sharp), 
\end{array}
\end{equation}
$I\supset J$ has a $(1,2)$-monomializing 
$\ell$-basis $(y_{3},y_{2})$ at $P$, 
$I\supset J$ has a $(1,2)$-monomializing 
$\ell$-basis $(z_{2},z_{3})$ 
\textup(resp. a $(1,2,2)$-monomializing $\ell$-basis $(z_{2,}z_{4},z_{3})$\textup) at 
$Q$, $\beta=z_{4}$ \textup(resp. 
$\beta \equiv z_{1}z_{3}+z^{2}_{2} \mod (z^{2}_{2},z_{3},z_{4})(z_{2},z_{3},z_{4})$\textup) 
if $k=1$ \textup(resp. $k\ge 2$\textup)\textup, where 
$k$ is the axial multiplicity of $Q$. Furthermore, 
there is an $\ell$-splitting 
\begin{equation}
\label{eq-gr2OJ712}
\textstyle
\gr^{2}(\OOO,J) \simeq 
\left(2P^\sharp\right) 
\toplus\left(-1+\frac{m-1}{2}P^\sharp+Q^\sharp\right). 
\end{equation}
\end{lemma}

\begin{proof}
Proof will be given in a few steps. 
First by Lemma \ref{(2.13.9)} 
we have $\qldeg(\MMM,P) \neq (m-1)/2$.

\begin{step}[{\cite[(2.13.10.1)]{Kollar-Mori-1992}}]
\label{(2.13.10.1)} 
\textbf{Claim:} $P$ is ordinary.

Assuming that $P$ is not 
ordinary, we will derive a contradiction. We may assume 
$\alpha \equiv y_{1}y_{3} \mod (y_{2},y_{3},y_{4})^{2}$ 
by Lemma \ref{(2.13.1)}. 
Thus $y_{2}$ and $y_{4}$ form an 
$\ell$-free $\ell$-basis of $\gr^{1}_{C}\OOO$ at $P$, 
and we may assume that they are $\ell$-free $\ell$-bases of $\LLL$ and $\MMM$, respectively because $\qldeg(\MMM,P) \neq (m-1)/2$. Hence $\MMM \simeq (-1+Q^\sharp)$. By the deformation $\alpha+ty^{2}_{4}$ \cite[(2.9.2)]{Kollar-Mori-1992}, 
see also Remark \ref{rem-def}, 
we may assume that $I\supset J$ has a $(1,2,2)$-monomializing $\ell$-basis $(y_{4},y_{2},y_{3})$ at $P$. We 
may further assume that $Q$ is an ordinary point by \cite[(2.9.2)]{Kollar-Mori-1992}. Hence 
$\LLL \simeq (\frac{m-1}2P^\sharp+Q^\sharp)$ and $\gr^{2,1}(\OOO,J) \simeq \MMM^{\totimes2}\totimes(P^\sharp) \simeq (-1+P^\sharp)$. 
Therefore, by \cite[(8.12)(ii)]{Mori-1988}
\[
\begin{array}{ll}
\gr^{1}(\omega,J) &\simeq \MMM\totimes\gr^{0}_{C}\omega \simeq (-1+\frac{m-1}2P^\sharp), 
\\[10pt]
\gr^{2,0}(\omega,J)&\simeq \LLL\totimes\gr^{0}_{C}\omega \simeq ((m-1)P^\sharp), 
\\[10pt] 
\gr^{2,1}(\omega,J)&\simeq \gr^{2,1}(\OOO,J)\totimes\gr^{0}_{C}\omega \simeq (-2+\frac{m+1}2P^\sharp+Q^\sharp), 
\\[10pt] 
\gr^{3,0}(\omega,J)&\simeq \gr^{2,0}(\omega,J)\totimes \MMM \simeq (-1+(m-1)P^\sharp+Q^\sharp), 
\\[10pt] 
\gr^{3,1}(\omega,J)&\simeq \gr^{2,1}(\omega,J)\totimes \MMM \simeq (-2+\frac{m+1}2P^\sharp). 
\end{array}
\]
Hence, $H^i(\gr^{1}(\omega,J))=0$, $i=1,\, 2$. 
>From the exact sequences 
\[
0\to \gr^{n,1}(\omega,J) \to \gr^{n}(\omega,J) \to 
\gr^{n,0}(\omega,J) \to 0, \quad n=2,\ 3
\]
we obtain 
$H^1(\gr^{2}(\omega,J))=H^1(\gr^{3}(\omega,J))=\CC$.
Finally, from the exact sequences
\eqref{eq-omegaJF} follows 
$H^i(\omega /F^{2}(\omega,J))=0$, $i=1,\, 2$, \ 
$H^1(\omega /F^{3}(\omega,J))=\CC$, and
$H^1(\omega /F^{4}(\omega,J))\neq 0$.
By \cite[Theorem (4.4)]{Mori-Prokhorov-2008} we have 
$V:=\Spec_X \OOO_X/F^4(\OOO,J)\supset f^{-1}(o)$.
Hence $2=(-K_X\cdot f^{-1}(o))\le (-K_X\cdot V)=6/(2m)$,
a contradiction.
Thus $P$ is ordinary as claimed. 
\end{step}

\begin{step}[{\cite[(2.13.10.2)]{Kollar-Mori-1992}}]
\label{(2.13.10.2)} 
\textbf{Claim:} $m\ge 5$.

Assume that $m=3$. Then $\qldeg(\MMM,P)=1$ 
because 
$\qldeg(\MMM,P) \equiv - \wt\ y_{3}\equiv 1$. 
This contradicts the original assumption that 
$\qldeg(\MMM,P) \neq(m-1)/2=1$.
Thus $m\ge 5$ as claimed.
\end{step}

\begin{step}[{\cite[(2.13.10.3)]{Kollar-Mori-1992}}]
\label{(2.13.10.3)} 
Since $\gr^{1}_{C}\OOO$ has an $\ell$-free $\ell$-basis $\{y_{2},y_{3}\}$ at $P$, the 
assertions on $\ell$-bases of $\LLL$ and $\MMM$ at $P$ follow. 
Therefore
$(y_3,y_2)$ is a $(1,2)$-monomializing $\ell$-basis for
$I \supset J$ at $P$ because
$I^\sharp=(y_3,y_2)$ and $J^\sharp=(y_3^2,y_2)$ at $P$.

Since $\gr^{1}_{C}\OOO$ has an $\ell$-free $\ell$-basis $\{z_{2},z_{3}\}$ (resp. $\{z_{2},z_{4}\}$) at $Q$ if $k=1$ 
(resp. $k\ge 2$), the 
assertions on $\ell$-bases of $\LLL$ and $\MMM$ at $Q$
follow from $\qldeg(\MMM,Q)=1$
(see Remark \xref{(2.13.8.1)}) possibly after a change of
coordinates. 
Thus \eqref{eq-IAIA-LM} is settled.

Assume $k=1$. Then $Q$ is ordinary, 
$I^\sharp=(z_2,z_3)$, and $J^\sharp=(z_2^2,z_3)$ at $Q$,
whence $(z_3,z_2)$ is a $(1,2)$-monomializing $\ell$-basis. In particular, 
$\gr^{2,1}(\OOO,J) \simeq \MMM^{\totimes2}$.

Thus we only have to show that $(z_{2},z_{4},z_{3})$ is a $ 
(1,2,2)$-monomializing $\ell$-basis of $I\supset J$ 
assuming $k\ge 2$. 
Hence $J^\sharp =(z^{2}_{2},z_{3},z_{4})$ and
$\beta \equiv z_{1}z_{3}+cz^{2}_{2} \mod J^\sharp I^\sharp$ for 
some $c \in {\CC}$. If $c=0$, then $z_{3} \in F^{3}(\OOO,J)$ and $\gr^{2,1}(\OOO,J) \simeq \MMM^{\totimes2}$, 
whence 
\[
\begin{array}{ll}
\gr^{2,0}(\omega,J)&\simeq \LLL\totimes\gr^{0}_{C}\omega \simeq (-1+(m-1)P^\sharp+Q^\sharp), 
\\[10pt] 
\gr^{2,1}(\omega,J)&\simeq \MMM^{\totimes2}\totimes\gr^{0}_{C}\omega \simeq (-2+\frac{m+3}{2}P^\sharp+Q^\sharp). 
\end{array}
\]
As in the Step \ref{(2.13.10.1)} we get
$H^{1}(\omega /F^{3}(\omega,J)) \neq 0$
which implies a contradiction.
Thus $c \neq0$ and 
the assertion on $\ell$-basis is proved. In particular, the assertion on $\beta$ follows. 
So if $k \ge 2$, then $c \not=0$ and $z_3$ is an $\ell$-free
$\ell$-basis of $\gr^{2,1}(\OOO,J)$ and
$\gr^{2,1}(\OOO,J)\simeq \MMM^{\totimes2}\totimes(Q^\sharp)$.
\end{step}

\begin{step}[{\cite[(2.13.10.4)]{Kollar-Mori-1992}}]
\label{(2.13.10.4)}
Hence by \eqref{eq-IAIA-LM},
there are two cases:
\[
\LLL=
\begin{cases}
\left(\frac{m-1}{2} P^\sharp+Q^\sharp\right)
\\
\left(\frac{m-1}{2} P^\sharp\right)
\end{cases}
\hspace{3pt}
\gr^{2,1}(\OOO,J)=
\left\{
\begin{array}{lll}
(-1+2P^\sharp)& \text{if $k=1$,}
\\
(-1+2P^\sharp+Q^\sharp)& \text{if $k\ge 2$,}
\end{array}\right.
\]
Thus from the exact sequence
\[
0 \to \gr^{2,1}(\OOO,J) \to 
\gr^{2}(\OOO,J) \to \LLL \to 0,
\]
we have
$\gr^{2}(\OOO,J) \simeq \OOO_{C}\oplus \OOO_{C}(-1)$ as $\OOO_{C}$-modules and
one of the following 
holds \cite[(2.8)]{Kollar-Mori-1992}:
\[
\gr^{2}(\OOO,J)=
\left\{
\begin{array}{llll}
\left(\frac{m-1}{2} P^\sharp+Q^\sharp\right)& 
\toplus& \left(-1+2 P^\sharp\right)&(*_1)
\\[5pt]
\left(2 P^\sharp+Q^\sharp\right)&
\toplus&\left(-1+\frac{m-1}{2} P^\sharp\right)&(*_2)
\\[5pt]
\left(\frac{m-1}{2} P^\sharp\right)& 
\toplus& \left(-1+2 P^\sharp+Q^\sharp\right)&(*_3)
\\[5pt]
\left(2 P^\sharp\right)&
\toplus&\left(-1+\frac{m-1}{2} P^\sharp+Q^\sharp\right)&(*_4)
\end{array}\right.
\]
Note also that 
$\gr_C^0\omega=(-1+\frac{m-1}{2}P^\sharp+Q^\sharp)$.
To determine 
the $\ell$-splitting of $\gr^{2}(\OOO,J)$, 
it is enough 
to disprove the $\ell$-isomorphisms $(*_1)$,
$(*_2)$, $(*_3)$ 
when $m\ge 7$, and $(*_1)$,
$(*_2)$ when $m=5$.
Then 
\[
\gr^{2}(\omega,J)=
\left\{
\begin{array}{llll}
\left((m-1)P^\sharp\right)
&\toplus&(-2+\frac{m+3}{2}P^\sharp+Q^\sharp)&(*_1)
\\[10pt]
\left(\frac{m+3}{2}P^\sharp\right)
&\toplus&(-2+(m-1)P^\sharp+Q^\sharp)&(*_2)
\\[10pt]
(-1+(m-1)P^\sharp+Q^\sharp) 
&\toplus& (-1+\frac{m+3}{2}P^\sharp)&(*_3)
\end{array}\right.
\]
Since 
$\gr^1(\omega, J)=\gr^1(\OOO,J)\totimes \omega
=\left(-1+\frac{m+1}{2} P^\sharp\right)$,
$H^i(\gr^1(\omega, J))=0$ for $i=0$, $1$.

In the first two cases $(*_1)$ and $(*_2)$ 
one has $H^1(\gr^{2}(\omega,J))\neq 0$.
As in the Step \ref{(2.13.10.1)} we get
$H^{1}(\omega /F^{3}(\omega,J)) \neq 0$
which implies a contradiction.
In the case $(*_3)$, one has
$H^i(\gr^2(\omega, J))=0$ for $i=0$, $1$,
and a computation similar to one in the Step \ref{(2.13.10.1)} shows
\[
\textstyle
\gr^{3}(\omega,J)
\simeq \gr^{2}(\omega,J) \totimes \MMM
\simeq
(0) \toplus\left(-2+\frac{m+5}{2}P^\sharp+Q^\sharp\right).
\]
If $m \ge 7$, again as in the Step \ref{(2.13.10.1)} we get
$H^{1}(\omega /F^{4}(\omega,J)) \neq 0$
which implies a contradiction.
Thus \eqref{eq-gr2OJ712} holds.
\end{step}
\end{proof}

\begin{lemma}[{\cite[(2.13.11)]{Kollar-Mori-1992}}]
\label{(2.13.11)}
We use the notation and assumptions of Lemma \xref{(2.13.10)}. 
Then
$H^{0}(\OOO(-K_{X}))=H^{0}(F^{2}(\omega {^*},J))$ and 
a general section $s$ of $H^{0}(\OOO(-K_{X}))$ 
induces a section $\ov{s}$ of $\gr^{2}(\omega {^*},J)$ such that

\begin{substatement}[{\cite[(2.13.11.1)]{Kollar-Mori-1992}}]
\label{(2.13.11.1)} 
$\ov{s}$ generates $\LLL\totimes\gr^{0}_{C}\omega {^*} \subset \gr^{1}_{C}\omega {^*}$ at $P$, and 
\end{substatement}

\begin{substatement}[{\cite[(2.13.11.2)]{Kollar-Mori-1992}}]
\label{(2.13.11.2)} 
if $m\ge 7$ then $\ov{s}$ is a global generator of $(0)$ 
in the $\ell$-splitting of \eqref{eq-gr2OJ712} 
\[
\textstyle
\gr^{2}(\omega {^*},J) \simeq (0) 
\toplus\left(-1+
\frac{m+5}{2}P^\sharp+Q^\sharp\right). 
\]
If $m=5$, the same assertion holds possibly after changing 
the $\ell$-splitting of $\gr^{2}(\omega {^*},J)$. 
\end{substatement}
\end{lemma}

\begin{proof}
We see $H^{0}(\OOO(-K_{X}))=H^{0}(F^{2}(\omega {^*},J))$ 
by $H^{0}(\gr^{0}(\omega {^*},J))=
H^{0}(\gr^{1}(\omega {^*},J)) =0$
(see Lemma \xref{(2.13.10)}). 
Let $D=\{y_{1}=0\}/\muu_{m} \in |-2K_{X}|$ and let 
$\Omega$ be an $\ell$-free $\ell$-basis of $\gr^{0}\omega$ at $P$. 
As in \ref{ppar-53} by
Corollary \xref{surj_on_omega},
$y_{2}/\Omega \in \OOO_{D}(-K_{X})$ 
lifts modulo $\OOO_{D^\sharp}(y_2,y_3)^2dy_2\wedge dy_3$
to a section of $H^{0}(F^{2}(\omega {^*},J))$. Since $y_{2}$ is a part of an $\ell$-free $\ell$-basis of 
$\gr^{2}(\OOO,J)$, we see that $\ov{s}$ is non-zero. 
If $m\ge 7$, 
then $\ov{s}$ must generate $(0)$ because 
$H^{0}(C,(-1+\frac{m+5}{2}P^\sharp+Q^\sharp))=0$. 
If $m=5$, we see as above
\[
H^{0}(\OOO(-K_{X})) 
\twoheadrightarrow \gr^{2}(\omega {^*},J)\otimes {\CC}(P)
\]
using $y^{2}_{3}/\Omega \in 
\OOO_{D}(-K_{X})$. 
Then general $s$ satisfies $\ov{s} \notin H^{0}(C,(Q^\sharp))$ in the $\ell$-splitting of $\gr^{2}(\omega {^*},J)$ 
and we have the same conclusion. 
\end{proof}

\begin{lemma}[{\cite[(2.13.12)]{Kollar-Mori-1992}}]
\label{(2.13.12)}
We assume the notation and assumptions of 
Lemma \xref{(2.13.10)}. 
In particular, we assume $m \ge 5$.
Then the case \xref{(2.2.3)} holds. 
\end{lemma} 

\begin{proof}
Let $s \in H^{0}(\OOO(-K_{X}))$ be a general section. If $m=5$, we 
change the $\ell$-splitting of $\gr^{2}(\OOO,J)$ for which 
Lemma \xref{(2.13.11)} holds. 
Depending on the value of $k$, we treat two cases.

\begin{case}[$k=1$, {\cite[(2.13.12.1)]{Kollar-Mori-1992}}]
\label{(2.13.12.1)} 
We claim that the image of $\ov{s}$ in $\gr^{1}_{C}\omega {^*}$ generates $\LLL\totimes\gr^{0}_{C}\omega {^*} \simeq (1)$ ($\subset \gr^{1}_{C}\omega {^*}$) at $P$ and $Q$ and vanishes 
at some point $R$ $(\neq P, Q$). Indeed, the generation at $P$ is proved in 
Lemma \xref{(2.13.11)}. If $\ov{s}$ does not generate $\LLL\totimes\gr^{0}_{C}\omega {^*}=\gr^{2,0}(\omega {^*},J)$ at $Q$, $\ov{s}$ is not a part of an $\ell$-free $\ell$-basis of $\gr^{2,0}(\omega {^*},J)$ at $Q$ 
because 
\[
\textstyle
\qldeg(\gr^{2,1}(\omega {^*},J),Q)=\qldeg
\left(\MMM^{\totimes2}\totimes\gr^{0}_{C}\omega {^*},Q\right) 
=1 \neq 0. 
\]
This
contradicts Lemma \xref{(2.13.11)} and our claim is proved. 

Then it is easy to see 
that $E_{X}=\{s=0\} \in |-K_{X}|$ is smooth outside of $P, Q$ and $R$. Moreover, $(E_{X},P) \simeq (y_{1},y_{3})/\muu_{m}(1,-1)$ and 
$(E_{X},Q) \simeq (z_{1},z_{2})/\muu_{2}(1,1)$. We choose 
coordinates at $R$ so that 
$(X,R)=(w_{1},w_{2},w_{3})\supset (C,R)=w_{1}$-axis, and $J=(w_{2},w^{2}_{3})$ at $R$. Using a generator $\Omega$ of $\OOO(K_{X})$ at $R$, we see 
$\Omega s \equiv uw_{1}w_{2} \mod (w_{2},w_{3})^{2}$ 
for some unit $u$
because $\ov{s}$ vanishes at $R$ to order 1. Since 
$\Omega s$ is a part of a free basis of 
$\gr^{2}(\OOO,J)$ at $R$, we have 
\[
\Omega s \equiv u w_{1}w_{2} 
+vw^{2}_{3} \mod (w_{2},w^{2}_{3})(w_{2},w_{3})
\]
for some unit $v$. Thus $(E_{X},Q)$ 
is an $A_{1}$ point and we are done in case $k=1$. 
\end{case}

\begin{case}[$k\ge 2$, {\cite[(2.13.12.2)]{Kollar-Mori-1992}}]
\label{(2.13.12.2)}
We see that the image of $\ov{s}$ in $\gr^{1}_{C}\omega {^*}$ 
generates $\LLL\totimes\gr^{0}_{C}\omega {^*} \simeq (Q^\sharp)$ outside of $Q$ by Lemma \xref{(2.13.11)}. Then $E_{X}=\{s 
=0\} \in |-K_{X}|$ is smooth outside of $P$ and $Q$, $(E_{X},P) \simeq 
(y_{1},y_{3})/\muu_{m}(1,-1)$. Using an $\ell$-free $\ell$-basis $\Omega$ of $\OOO(K_{X})$ at $Q$, we 
see the image of $\ov{s}$ in $\gr^{1}_{C}\omega {^*}$ is $uz_{1}z_{4}/\Omega$ at $Q$, where $u$ is a unit. Since $s$ is a 
part of an $\ell$-free $\ell$-basis of $\gr^{2}(\omega {^*},J)$ at $Q$, we have 
$\Omega s \equiv uz_{1}z_{4}+
vz_{3} \mod J^\sharp I^\sharp$ 
at $Q$ for some unit $v$. 
Eliminating $z_{3}$, we see $(E_{X},Q) \simeq (z_{1},z_{2},z_{4};\ov{\beta})/\muu_{2}(1,1,0;0)$, 
where $\ov{\beta}$ satisfies 
$\ov{\beta} \equiv z^{2}_{1}z_{4}+
z^{2}_{2} \mod (z^{2}_{2},z_{4})(z_{2},z_{4})$ 
and $\ord \ov{\beta}(0,0,z_{4})=k$. 
Then we can apply 
Computation \xref{(2.13.6)}. 
\end{case}
\end{proof}

\begin{remark}
\label{rem-on-2.2.3}
We note that the case \xref{(2.2.3)} 
(\cite[(2.2.3)]{Kollar-Mori-1992}, \cite{Mori-2007}) 
comes out of two sources:
Lemma \xref{(2.13.4)} where $k \ge 2$, $m\ge 3$ and $Q$ is of type
$cA/2$, $cAx/2$ or $cD/2$, 
and Lemma \xref{(2.13.10)} where 
$m\ge 5$ and $Q$ is of type $cA/2$.

We note that Lemma \xref{(2.13.4)} assumes the case \xref{(2.13.3.2)},
where $(X^{\sharp},Q^{\sharp})$ is not smooth by $\ell(Q) >0$ and hence
the axial multiplicity $k \ge 2$.
\end{remark}

Thus the proof of Theorem \xref{(2.2)} is completed in the case \type{IA}$+$\type{IA}.

\subsection*{Acknowledgments}
The work was carried out at Research Institute for Mathematical
Sciences (RIMS), Kyoto University. The second author would like to
thank RIMS for invitation to work there in February 2008, for
hospitality and stimulating working environment.

\end{document}